\documentclass[11pt,leqno]{amsproc}
\usepackage{amssymb,amsmath}
\usepackage{subfig}
\usepackage[english]{babel}
\usepackage{graphicx}
\usepackage{tabularx}
\usepackage{enumerate, hyperref}

\usepackage{tikz}
\usetikzlibrary{shapes.geometric,calc}

\newcommand\Z{\mathbb Z}

\def\R{\mathbb R}

\newcommand\sinc{\operatorname{sinc}}

\newcommand\rect{\operatorname{rect}}
 \newcommand\Span{\operatorname{Span}}
 \newcommand\supp{\operatorname{supp}}

\def\B{\mathcal B}
\def\C{\mathbb C}

\newtheorem{theorem}{Theorem}
\numberwithin{equation}{section}
\newtheorem{lemma}[theorem]{Lemma}
\renewcommand\r{\rangle}
 \renewcommand\l{\langle}
\newcommand\cal{\mathcal}

\newtheorem{corollary}[theorem]{Corollary}

\title[p-Riesz bases in quasi shift invariant spaces ] { p-Riesz bases in quasi shift invariant spaces}
\author{ Laura De Carli, Pierluigi Vellucci}
\address{ Laura De Carli: Florida International Univ., Univ. Park, Miami (FL)}
\email{decarlil@fiu.edu}
\address{Pierluigi Vellucci: Dept. of Economics, Roma Tre University, via Silvio D'Amico 77, 00145 Rome, Italy.}
\email{pierluigi.vellucci@uniroma3.it}

\begin{document}
\maketitle

 \begin{abstract}
   Let $ 1\leq p< \infty$ and  let $\psi\in L^{p}(\R^d)$. We study $p-$Riesz bases of  quasi shift invariant spaces $V^p(\psi;Y)$.
  \end{abstract}

\section{Introduction}
\label{sec:intro}

Let $ 1\leq p< \infty$ and  let $\psi\in L^{p}(\R^d)$.   We consider the shift invariant space
$
V^p(\psi)= \overline{\Span \{\tau_k \psi\}_{k\in\Z^d  } },
$
where $\tau_{  s}f(x)=f(x+  s)$ is the translation and ``bar'' denotes the closure in $L^p(\R^d)$.
Shift-invariant spaces appear naturally in signal theory and in other branches of applied sciences.
In  \cite{AST} \cite{GS} and in the recent preprint \cite{HL}
   {\it quasi-shift invariant spaces}  of functions are considered.
Given $ X = \{x_j\}_{j\in\Z^d} $, a  countable and discrete\footnote{ A countable set $X\subset \R^d$ is {\it discrete}  if for every $x_j\in X$ there exists $\delta_j>0$ such that $|x_j-x_k|_2>\delta_j$ for every $k\ne j$.}  subset of $\R^d$  and a function
$\psi\in L^p(\R^d)$, we  let
\begin{equation}\label{def-quasi}
V^p ( \psi;\,X) =  \overline{ {\rm Span}\{\tau_{x_j}\psi \} }    .\end{equation}
Thus, $V^p(\psi)=V^p(\psi; \Z^d)$.
Quasi-shift invariant spaces are also called  {\it Spline-type  spaces}    in  \cite{F}, \cite{FMR} \cite{FO}, \cite{Ro}.

Following   \cite{AST},  \cite{CS},
 we say that  the translates $\{\tau_{x_j}\psi \}_{j\in\Z^d}$ form a p-Riesz basis in $V^p ( \psi;\,X) $ if there exist constants $A,\ B>0$ such that, for every   finite set of coefficients $\vec d=\{d_j\}  \subset\C $,
\begin{equation}\label{E-p-basis-2}
A \|\vec d\|_{\ell^p} \leq \| \sum_j d_j\tau_{ x_j} \psi  \|_{p} \leq B\|\vec d\|_{\ell^p}.
\end{equation}
Here and throughout the paper, we have let $\|f\|_p =\left(\int_{\R^d} |f(x)|^p dx\right)^{\frac 1p}$ and   $\|\vec c\|_{\ell^p}=(\sum_{j\in\Z^d} |c_j|^p)^{\frac 1p}$. If   $x=(x_1, ...,\, x_d), \ y=(y_1, ...,\, y_d) \in\R^d$,  we   will often let $ x\cdot y= x_1y_1+...+x_dy_d $ and $|x  |_2
=(  x\cdot x )^{\frac 12}$. We will also let  $|x|_\infty= \sup_{1\leq j\leq d}|x_j|$.

If \eqref{E-p-basis-2}  holds,  then it is possible to prove    that
\begin{equation}\label{d-Vp} \footnote{ A proof of this ientity was kindly provided to us by K. Hamm}
V^p(\psi; X)= \{ f= \sum_{k\in\Z^d}d_k \tau_{x_k}\psi(x), \ \vec d \in \ell^p\  \}
\end{equation}
and the sequence $\{d_k  \}_{k\in\Z^d}$ is uniquely determined by $f$.

$p-$Riesz bases
%in shift-invariant spaces are well studied and understood.
%They
allow a stable reconstruction of functions  in $V^p(\psi; X)$; when  $X=\Z^d$ and  $\B=\{\tau_j\psi\}_{j\in\Z^d}$ is a $p-$Riesz basis  of $V^p(\psi)$,
the  coefficient $d_j$ in \eqref{d-Vp} can be expressed in an unique way in terms of the functions in the dual basis  of $\B$. See \cite{S}, \cite{BR} and also \cite{AS}  for   explicit reconstruction formulas.

When $\psi$ has compact support, it is   known  (see e.g. \cite[Prop. 1.1]{AS}, \cite{JM}, \cite{R}) that $\B$ is a $p-$Riesz basis in $V^p(\psi)$ if and only if
 $   \sum_{m\in\Z^d} |\hat\psi(y+m)|^2\ne 0 $ for every  $y\in [-\frac 12, \frac 12)^d$ and every $m\in\Z^d$. See also Lemma \ref{L-eq-cond-V2} in Section 2.

We  have denoted with  $\hat\psi(y)=\int_{\R^d} e^{2\pi i x\cdot y}f(x)dx$   the Fourier transform of $\psi$.
   The proof of the aforementioned  result relies on the lattice structure of $\Z^d$ and  on standard Fourier analysis technique and does not easily generalize  to other sets of translations.

\medskip
Let $\psi\in L^p(\R^d)$, $1\leq p< \infty$, and let $X=\{x_j\}_{j\in\Z^d}$ be a   discrete set of $\R^d$.   It is natural to consider the following    problem:

\medskip
\noindent
{\bf Problem 1.} {\it  Let  $\B_X=\{\tau_{x_j}\psi \}_{j\in\Z^d}$ be a $p-$Riesz  basis for $V^p ( \psi; X ) $; can we find $\delta>0$  so that, for every   $Y=\{y_j\}_{j\in\Z^d}\subset \R^d$ with  $  \sup_j|y_j-x_j|_2 <\delta $, the set $\B_Y=\{\tau_{y_j}\psi \}_{j\in\Z^d}$  is  a $p-$Riesz basis for $V^p ( \psi;\, X )$? }

\medskip

This problem cannot be solved if $\psi$  has compact support.  For example, let  $\psi(x)=\rect(x)$  be the characteristic function of the interval $[-\frac 12, \frac 12)$ and  let $X=\Z$;     let $Y=\{y_n\}_{n\in\Z}$ be such that   $y_0=\delta>0$
  and $y_n=n$ when $n\ne 0$.  All functions in $ V^p(\rect;\,Y)$ vanish in the interval $[-\frac 12, -\frac 12+\delta]$ and so $ V^p(\rect;\, Y)\ne  V^p(\rect)$.

  We prove in Section 3 that Problem 1  can be solved when $p=2$ and $\psi$ is {\it band-limited}, i.e.,  when the Fourier transform of $\psi$ has compact support. See also Section 5 for more remarks and comments on problem 1
%We denote with   $ PW_D$  the Paley-Wiener space of functions whose Fourier transform is supported in $D$ (see e.g. %\%cite{He10}, p. 270).

 \medskip

We are concerned with the following problem:

\medskip
\noindent

{\bf Problem 2.} {\it  With the notation of Problem 1:  let $\B_X $ be a $p-$Riesz  basis for $V^p ( \psi;\, X ) $ and let $Y=\{y_n\}_{n\in\Z^d}$  that satisfies   $\sup_n|y_n-x_n|_2<\delta$; is   $\B_Y $   a $p-$Riesz  basis for $V^p ( \psi;\, Y )$ whenever $\delta$ is sufficiently small?
 }

\medskip
It is proved in \cite{FMR}  that Problem 2 has always solution when $X$ is {\it relatively separated}, i.e., when $X= X_1\cup...\cup X_k$, with $X_j=\{x_{j,n}\}_{n\in\Z^d} $ and $\inf_{n\ne m} |x_{j,n}-x_{j,m}|_2 >0$.

In Section 2 we prove the following theorem.

\begin{theorem}\label{C-PW}  Suppose that  that $\psi$ is in the Sobolev space $ W^{1,p}(\R^d) $, with $1<p<\infty$,  and that  $\{\tau_{x_j}\psi\}_{j\in\Z^d}$ is a $p-$Riesz  basis of $V^p(\psi; X)$. For every $j\in\Z^d$ there exists $\delta_j>0$ such that  $\{\tau_{y_j}\psi\}_{j\in\Z^d}$ is a $p-$Riesz  basis of $V^p(\psi;\, Y)$ whenever $|x_j-y_j|_2<\delta_j$.
\end{theorem}

  We recall that    $W^{1,p}(D)$   is the space of   $L^p(D)$  functions  whose partial distributional derivatives are also in $L^p(D)$ and that
$W^{1,p}_0(D)$ is the closure    of $ C^\infty_0(D)$ in $W^{1,p}(D)$.
%We denote with  $L^p_0(D) $  the space of the $L^p$ functions with compact support in $D$.
%When $D\subset\R$,  functions in $W^{1,p}(D) $ are continuous. See e.g. \cite{B}  for definitions and results on Sobolev spaces.

\medskip
When $X$ is not relatively separated   the  $\delta_j$'s  in Theorem \ref{C-PW} may  not have a positive lower bound, but  we can still solve   Problem 2  in the cases considered in Theorems \ref{T-non-bandlimited} and \ref{T-stab-Vpsi-bis} below.

 \begin{theorem}\label{T-non-bandlimited} Assume that  $\psi \in L^1(\R^d)\cap L^2(\R^d)$ satisfies
 \begin{equation}\label{e-amalgam-psi}
 0<c=\sum_{ k\in\Z^d} \inf_{x\in [0,1)^d}  | \hat \psi(x+  k)| ^2< \sum_{k\in\Z^d}  \sup_{x\in [0,1)^d}  | \hat \psi(x+k)| ^2 =C<\infty.\end{equation}
Then,  Problem 2 can be solved when $p=2$ and  $\{e^{2\pi i x_n\cdot x}\}_{n\in\Z^d}$ is a Riesz basis in $L^2([0,1)^d)$.
\end{theorem}

\medskip
We recall that  the {\it amalgam space} $W(L^\infty,\, \ell^q)$ is the set of measurable functions $f:\R^d\to\C$ for which
 $     ||f||_{W(L^\infty,\, \ell^q)}=\left(\sum_{n\in\Z^d} \sup_{x\in [0,1)^d} |f(x+n)|_2^q \right)^{\frac 1q}<\infty.$   The   amalgam space  $W(L^r,\, \ell^q)$ can be defined in a similar manner.

 The assumption  \eqref{e-amalgam-psi} implies that $\hat\psi$ is in the {\it amalgam space} $W(L^\infty,\, \ell^2)$.

 \medskip
 From Theorem \ref{T-non-bandlimited} follows that   if  $\psi  $ satisfies
 \eqref{e-amalgam-psi}, then  $ \{\tau_{y_k}\psi\}_{k\in\Z^d}$ is a $2-$Riesz  basis of $V^2(\psi, Y)$ whenever $|k-y_k|_\infty<\frac 14$. See the remark after the proof of Theorem \ref{T-non-bandlimited} in Section 4.

Exponential  Riesz bases in $L^2(0,1)$ are completely understood and classified \cite{Pavlov}. To the  best of our knowledge, no such characterization exists for exponential bases on $L^2((0,1)^d)$ when $d>1$.

 \medskip
For our next theorem we consider  $\psi$   in the Sobolev space $W^{1,p}_0(\R^d)$;
 we denote with  $\partial_j\psi=\frac{\partial\psi}{\partial x_j}$  the    partial  derivative (in distribution sense)  of $\psi$ and we let $\nabla \psi= (\partial_1\psi,\, ...,\, \partial_d\psi)$ be the gradient of $\psi$. Let $Y=\{y_k\}_{k\in\Z^d}$  and  $$L=\sup_{k\in\Z^d}|{ y_{k}}-k|_2<\infty.$$ We prove the following

 \begin{theorem}\label{T-stab-Vpsi-bis}
 Let $D= (a_1,\,b_1)\times...\times (a_d, \,b_d)$ be a bounded rectangle in $\R^d$. Let
 $\psi\in W^{1,p}_0(D)$, with $1\leq p<\infty$, and let $\{\tau_k \psi \}_{k\in\Z^d}$ be a $p-$Riesz  basis of $V^p(\psi)$ with frame constants $0<A\leq B<\infty$.
If   \begin{equation}\label{e-1}
C=L\sum_{j=1}^d (1+[b_j-a_j +L])^{p-1}  \|\partial_j\psi \|_p^p< A,  \end{equation}   the set $\{\tau_{y_k}\psi\}_{k\in\Z^d}$ is a $p-$Riesz  basis of $V^p(\psi;\, Y )$ with constants $B+C$ and $A-C$.
\end{theorem}

The proofs  of Theorems \ref{T-non-bandlimited} and   \ref{T-stab-Vpsi-bis} are   in Section 3.

\medskip
Our Theorem \ref{T-stab-Vpsi-bis}  can be compared to \cite[Theorem 3.5]{FMR}.  In this theorem it is assumed that $|\nabla(\psi)|$ is in the  amalgam space  $W(L^\infty,\, \ell^1)$,
%i.e., that  $ ||\nabla \psi||_{W(L^\infty,\, \ell^1) }=\sum_{n\in\Z^d} \sup_{x\in [0,1)^d} |\nabla  \psi(x+n)|_2 <\infty$; it is also assumed that
and that $\inf_{x\in [0,1)^d}\sum_{k\in\Z^d}   |\nabla  \psi(x+k)|_2 >0$.

In the aforementioned theorem  is proved  that $\{\tau_{y_k}\psi\}_{k\in\Z^d}$ is a  Riesz  basis of $V^2(\psi;\, Y )$
if $C'=L^2(1+2L)^{2d}||\nabla \psi||_{W(L^\infty,\, \ell^1)}^2<A$.  Generalizations to functions for which $|\nabla(\psi)|$ is in the {  amalgam space} $W(L^q,\, \ell^1)$, with $q>d$ are also possible (see Remark 3.2 in \cite{FMR}).

 Our Theorem \ref{T-stab-Vpsi-bis} reduces to \cite[Theorem 3.5]{FMR}  when $p>d$ and  $\psi$  has  compact support.  For example,
when $\psi$ has support in $[0,1)^d$, the   norm in $W(L^p, \ell^1)$ reduces to   $||\nabla  \psi ||_p $.
The  constant $C$ in Theorem \ref{T-stab-Vpsi-bis}  may be smaller than $C'$ defined above  when the support of $\psi$ is small.
%so  Theorem 3.5 in \cite{FMR}  does not compare well with ours.

\medskip
Theorem \ref{T-stab-Vpsi-bis} does not apply when $\psi=\rect$ or when $\psi$ is a step function;
For $J\ge  1$, we let $ {\cal S}_J =\left\{s(t)=\sum_{|j|\leq J} s_j \rect( t-j),\, s_j\in\R\right\}\,  $. We let    $p'=\frac{p}{p-1}$ be the dual exponent of $p$.
The following theorem is proved in Section 4.2.

{\begin{theorem}\label{T-step}
 Assume that $g\in {\cal S}_J$ and that $ \{\tau_k g\}_{k\in\Z} $ is a $p-$Riesz basis for $V^p(g)$, %with $1\leq p <\infty$,
 with frame constants $A$ and $B$.  If
$$2^p L J \, \|g\|_{p'}^{p } <A\, ,$$
the sequence $\{\tau_{y_k}g\}_{k\in\Z} $ is a Riesz basis for $V^p(g; Y)$.
\end{theorem}

\noindent
{\it Acknowledgement.} We are grateful to the anonymous referee of this paper for her/his thorough reading of our manuscript and for providing suggestions that have   improved  the quality of our work.

We also wish to thank K. Hamm for providing a proof of the identity \eqref{d-Vp}  for $p\ne 2$.

\section{ Preliminaries}

\subsection{Notation}
  We denote with $\l f ,\, g \r =\int_{\R^d} f(x)\bar g(x)dx$ and $\|f\|_2= \sqrt{\l f, f\r}$ the standard inner product and norm in $L^2(\R^d)$.  For a given $p\in\R^d$ and $\delta>0$, we let  $B(p,\delta ) =\{x\in\R^d : | x-p|_2<\delta\}$.

 We let  $\rect(x)= \chi_{[-\frac 12, \frac 12)}(x)$ be the characteristic function of the interval $[-\frac 12, \frac 12)$ and
 $\beta^{s }=\rect^{(s+1)}(x)=\rect*... *\rect(x) $ be the $s+1-$times iterated convolution of    $\rect  $. The function $\beta^{s}(x)$, a piecewise polynomial function of degree $s $, is a  {\it B-spline} of order $s $.
See  \cite{Sch2}, where the B-splines were first introduced, and   \cite{PBP}, \cite{UAE}  and the references cited there.

\subsection{$p-$ Riesz bases}

Recall that a  Schauder basis  in a separable Banach space $V$ is a linearly independent set   $\B=\{v_j\}_{j\in\Z}$ such that:  $\overline{span(\B)}=V$, and
there exists a sequence of bounded linear functions $f_j : X\to\C$  (the {\it functional coefficients of the basis}) such that
$x=\sum_j f_j(x) v_j$ for every $x\in V$.

Following  \cite{CCS}, \cite{CS}  and  \cite{AST}, we say that $\B$ is a $p-$Riesz  basis of $V$, with $1\leq p< \infty$, if $\overline{ \Span(\B) }=V$, if  every series  $\sum_n a_nv_n$ converges in $V$ when  $\vec a=(a_n)_{n\in\Z}\in \ell^p$ and
if   there exist constants $A,\ B>0$ such that, for every finite sequence of coefficients $\{d_j\}_{j\in\Z} \subset\C $, we have
$$
A \|\vec d\|_{\ell^p} \leq \| \sum_j d_jv_j    \|_{p} \leq B\|\vec d\|_{\ell^p}.
$$
Every $p-$Riesz basis is a Schauder basis. As mentioned in the introduction, when $V=V^p(\psi)$  and $\psi$ has compact support,
the  functional coefficients of the basis $\{\tau_k\psi\}_{k\in\Z}$  of $V^p(\psi)$ can be  written  in terms of the dual functions of the basis.

The following results are well known (see e.g. \cite[Prop. 1.1]{AS}, \cite{JM}, \cite{R}).

\begin{lemma}\label{L-eq-cond-V2}  a) Let $\psi\in L^p_0(\R^d) $.  The set $\B=\{\tau_k\psi\}_{k\in\Z^d}$ is a $p-$Riesz basis in $V^p(\psi)$ if and only if
\begin{equation}\label{cond-sum} \sum_{m \in\Z^d} |\hat \psi(y+m)|^2\ne 0   \quad \mbox{for every  $y\in  [-\frac 12, \frac 12)^d$}.\end{equation}

b) If   $\psi \in W(L^\infty, \ell^1)$ is continuous and if    $\B$  is  Riesz basis in $V^2(\psi)$, then
$\B$  is   a $p$-Riesz basis in $V^p(\psi)$ for every $1\leq p<\infty$.
\end{lemma}

\begin{proof} For the convenience of the reader we prove that if   $  \psi \in L^2_0(\R^d)$,
  $\B$ is a  Riesz basis  of $V^2(\psi)$ with constants $0< A\leq B<\infty$ if and only if  the following inequality holds for every $ y\in  Q=[-\frac 12, \frac 12)^d .$
\begin{equation}\label{e-pointwise-FT}
A\leq    \sum_{m \in\Z^d} |\hat \psi(y+m)|^2 \leq B.
\end{equation}
%
%We let $Q=[-\frac 12, \frac 12)^d$ for simplicity of notation.
 We can  verify (using e.g. the Poisson summation formula)  that the function in \eqref{cond-sum} is continuous in $\overline Q$, and so \eqref{e-pointwise-FT} is equivalent to \eqref{cond-sum}.
 %Thus, if $\psi \in C_0(\R^d)$, we have that
  %$\B$ is  a Riesz basis in $V^2(\psi)$ if and only if \eqref{cond-sum} holds  if and only if
  %$\B$ is  a $p-$Riesz basis in $V^p(\psi)$.

 Let $\{c_k\}_{k\in\Z^d}\subset\C $ be a finite set of coefficients such that $\sum_k|c_k|^2=1$.
 The Fourier transform of  $f=\sum_{k\in\Z^d}c_k\tau_k\psi$  is
$$
\hat f(y)=\hat\psi(y)\sum_{k\in\Z^d}c_k e^{2\pi i y\cdot k}  =  \hat\psi(y)M(y).
 $$
and by Plancherel's theorem
$ \displaystyle
\|f\|_2^2= \|\hat f\|_2^2=\sum_{m \in\Z^d} \int_{m+[-\frac 12, \frac 12)^d } |\hat f(y)|^2 dy
 =
\sum_{m \in\Z^d}\int_{Q } |\hat \psi(y+m)|^2 |M(y)|^2 dy  =\int_{Q } |M(y)|^2 \sum_{m \in\Z^d} |\hat \psi(y+m)|^2  dy.
$

Let $g= \sum_{m \in\Z^d} |\hat \psi(y+m)|^2 $; if \eqref{e-pointwise-FT} holds, from  $||f||_2^2= \int_Q |M(y)|^2 g(y)dy $  and $\int_{Q } |M(y)|^2dy =\sum_k |c_k|^2=1$ follows that   $A\leq ||f||_2^2\leq B$.

 Conversely, from
  $A  \sum_k |c_k|^2\leq \|f\|_2^2\leq  B \sum_k |c_k|^2  $ and    the above considerations, follows that
\begin{equation}\label{e3}
A \|M^2 \|_{L^1(Q )}  \leq \int_{Q } |M(y)|^2 g(y)dy \leq B \|M^2 \|_{L^1(Q )} .
\end{equation}
Every non-negative $h\in L^1(Q )$ can be written as  $ h= |M|^2$, with   $M\in L^2(Q )$.
The dual of $L^1(Q )$ is $L^\infty(Q )$   and so $\|g\|_{L^\infty(Q )} = \sup_{\|h\|_{L^1(Q )}=1}\int_{Q } f(y) g(y)dy$. From \eqref{e3} follows that
$ A\leq \|g\|_{L^\infty(Q )}   \leq B
$
as required.

\end{proof}

We will use the following  Paley-Wiener type result.
\begin{lemma}\label{L-PW} Let $X$, $Y\subset\R^d$ be countable and discrete.  Suppose that $\{\tau_{x_j}\psi\}_{j\in\Z^d}$ is a $p-$Riesz basis of $V^p(\psi; X) $   with constants $A\leq B$.
If the  inequality
$$\left\Vert\sum_j a_j(\tau_{x_j}\psi-\tau_ {y_j}\psi)\right\Vert_p^p\leq C\sum_n|a_n|^p
$$
 holds for all finite sequences $\{a_n\}_{n\in\Z^d}\in\C$ with a constant  $C<A$, the sequence
 $\{\tau_{y_j}\psi\}_{j\in\Z^d}$ is a $p-$Riesz basis of $V^p(\psi;\, Y )$ with constants $B+C$ and $A-C$.
\end{lemma}
\begin{proof}
Assume that $\sum_n|a_n|^p =1$; we have:
$$
\left\Vert\sum_j a_j \tau_{y_j}\psi \right\Vert_p\leq \left\Vert\sum_j a_j(\tau_{x_j}\psi-\tau_ {y_j}\psi)\right\Vert_p+\left\Vert\sum_j a_j \tau_{x_j}\psi\right \Vert_p
\leq C+B
$$
and
$$
\left\Vert\sum_j a_j \tau_{y_j} \psi\right\Vert_p\ge \left\Vert\sum_j a_j \tau_{x_j} \psi\right\Vert_p-\left\Vert\sum_j a_j(\tau_{x_j}\psi-\tau_ {y_j}\psi)\right\Vert_p
\ge A-C.
$$
\end{proof}

\begin{proof}[Proof of Theorem \ref{C-PW}  ]
Assume $p\in (1, \infty)$ and $\sum_j|a_j|^p=1$. Let   $p'=\frac p{p-1}$ is the dual exponent of $p$  and let  $\{\delta_j\}_{j\in\Z^d}$ be a sequence of positive constants such that
$\sum_j |\delta_j|^{p'} <\infty$,
We recall that, when $1<p<\infty$, a function  $f\in L^p(\R^d)$ is in  the Sobolev space $W^{1, p}(\R^d)$  if and only if there is a constant $c>0$  that depends on $f$ but  not on $\delta$, such that
\begin{equation}\label{e-ineq-sob}
\omega_p(\delta, f)=\sup_{|t|<\delta} ||\tau_t f -f ||_p \leq c \delta
\end{equation}
for every $\delta>0$. Furthermore, one can choose $c=\|\nabla f\|_p $. See e.g. Prop. 9.3 in \cite{B}.
%Assume $|x_j-y_j|<\delta_j$.
By \eqref{e-ineq-sob} and H\"older'
s inequality,   $$
  \left\Vert\sum_j a_j (\tau_{x_j} \psi - \tau_{y_j}\psi)\right\Vert\leq \sum_j |a_j|\, \|\tau_{x_j} \psi - \tau_{y_j}\psi\|_p \leq \sum_j |a_j|\, \delta_j
 $$
 $$\leq c \left( \sum_j |a_j|^p\right)^{\frac 1p}\left( \sum_j |\delta_j|^{p'}\right)^{\frac 1{p'}}=
 c\left( \sum_j |\delta_j|^{p'}\right)^{\frac 1{p'}}.
 $$
 We can chose the $\delta_j$ so small that $ c\left( \sum_j |\delta_j|^{p'}\right)^{\frac 1{p'}} <A$  and    use Lemma \ref{L-PW} to complete the proof.
 %
 %The cases when $p=1$ and $p=\infty$ are easy to prove.
\end{proof}

\section{Problem 1 $(p=2)$}

In this section we prove  that Problem 1   can be solved when $p=2$ and  $\hat\psi $ has compact support.

\begin{theorem}\label{T-band-lim} Let $\psi\in L^2(\R^d)$. Assume that $\hat\psi $ has compact support and   that there exist constants $c,\ C>0$  such that  $$c\leq\inf_{x\in \R^d}|\hat \psi(x)|\leq \sup_{x\in \R^d}|\hat \psi(x)|\leq C.$$
Let $\{\tau_{x_j}\psi\}_{j\in\Z^d}$ be  a Riesz basis in $V^2(\psi, X)$.
 There exists $\delta >0$ such that   if    $Y=\{y_j\}_{j\in\Z^d}\subset\R^d$ satisfies   $\sup_j|x_j-y_j|<\delta $,   then also $\{\tau_{x_j}\psi\}_{j\in\Z^d}$ is a Riesz basis of $V^2(\psi)$.
\end{theorem}

\begin{proof}
 Let $D=\supp(\hat\psi)$.
 When $p=2$, Plancherel theorem implies that the set $\{\tau_{x_j}\psi\}_{j\in\Z^d}$ is a Riesz basis in $V^2(\psi)$ if and only if the set $\{e^{2\pi i x_j\cdot x}\}_{j\in\Z^d}    $   is a   Riesz basis on $L^2(\R^d,\ \hat\psi\,dx)$.
 Our assumptions on $\hat\psi$ imply that the norm on $ L^2(\R^d,\ \hat\psi\,dx)$ is equivalent to the norm on $ L^2(D  )$ and that $\{e^{2\pi i x_j\cdot x}\}_{j\in\Z^d}$ is an exponential Riesz basis on $L^2(D)$.
 Exponential Riesz bases on bounded domains  of $\R^d$ are stable under small perturbations (see \cite{PW} and also Section 2.3 in \cite{KN});    we can  find $\delta>0$ such that, if    $Y=\{y_j\}_{j\in\Z^d}\subset\R^d$ satisfies   $\sup_j|x_j-y_j|<\delta $,   then also $\{e^{2\pi i y_j\cdot x}\}_{j\in\Z^d} $ is a   Riesz basis on $L^2(D)$ and hence also in $L^2(\R^d,\ \hat\psi\,dx)$.
 \end{proof}

\noindent
{\it Example.} Let  $d=1$  and let $\psi=\sinc(x)=\frac{\sin(\pi x)}{\pi x}$; the Fourier transform of $\tau_k\psi(x)=\sinc (x-k)$ is $e^{2\pi i kx}\rect(x)=e^{2\pi i kx}\chi_{[-\frac 12, \frac 12)}(x)$,  and so   $V^2(\psi)$ is isometrically isomorphic to $ \overline{\Span\{e^{2\pi i jx}\rect(x)\}_{j\in\Z^d}} =  L^2(-\frac 12, \frac 12)$.
By Kadec's    theorem (\cite{K}, \cite{Y}) if  $Y=\{y_n\}_{n\in\Z^d}\subset\R$ is such that $ \sup_n|y_n-n| \leq \delta  < \frac 14$,   the set $\{e^{2\pi i y_n x}\}_{n\in\Z^d}$ is still a Riesz basis of $L^2(-\frac 12,\frac 12)$  and therefore,   the set $\{\sinc(x-y_n)\}_{n\in\Z^d}$ is a Riesz basis  for $V^2(\sinc)$.  Thus,
 $V^2(\sinc; \, Y) =V^2(\sinc)$.

 \medskip
Things are not so clear when $p\ne 2$. For example,  the trigonometric system $\B=\{e^{2\pi i n x}\}_{n\in\Z^d}$  is a  Schauder basis in  $L^p(-\frac 12, \frac 12)$  when $1<p<\infty$,  but it is not  a $p-$Riesz  basis and  the previous example cannot be generalized in an obvious way.  Stability results for  the Schauder basis $\B$  in   $L^p(-\frac 12, \frac 12)$ are proved in
  \cite{Russo} and in \cite{Sed16}.

\section{Problem 2}

In this section we prove Theorems \ref{T-non-bandlimited} and \ref{T-stab-Vpsi-bis}.
Let $X=\{x_n\}_{n\in\Z^d}$  and  $\B =\{e^{2\pi i x\cdot x_n}\}_{n\in\Z^d}$. We first prove  the following

\begin{lemma}\label{L-bases-amal}
Let $\psi\in L^2(\R^d)\cap L^1(\R^d)$ be as in  \eqref{e-amalgam-psi}; if $\B$ is a Riesz  basis in $L^2([0,1)^d)$ with constants $A_1$ and $B_1$  then   $\{\tau_{x_n}\psi \}$  is a Riesz basis of $V^2(\psi, X)\}$ with constants $A=A_1c$ and $B=B_1 C$.

\end{lemma}

 \begin{proof}
For $k\in\Z^d$, we let $c_k=\inf_{x\in (0,1]^d}| \hat \psi(x+  k)|^2$ and $C_k=\sup_{x\in (0,1]^d}| \hat \psi(x+  k)|^2$.  Let $\{d_j\} $ be a   finite set of complex coefficient such that $\sum_j |d_j|^2=1$.   Since $\B$ is a Riesz basis in $L^2((0,1]^d)$, for every given $k\in\Z^d$ we have that
$$A_1\leq \left\Vert \sum_n  d_n e^{-2\pi i x_n \cdot k} \, e^{2\pi i x_n \cdot y} \right\Vert_{L^2((0,1]^d)}^2\leq B_1
.$$
From this inequality  follows at once that
$$c_k A_1\leq \left\Vert \sum_n d_n e^{-2\pi i x_n \cdot     k}e^{2\pi i x_n \cdot y} \hat\psi(.-k)\right\Vert_{L^2((0,1]^d)}^2\leq C_kB_1.
$$
With $c=\sum_{k\in\Z^d} c_k$ and  $C=\sum_{k\in\Z^d} C_k = ||\psi||_{W(L^\infty,\, \ell^2)}^2$, we have
$$A_1 c\leq \sum_{k\in\Z^d}\left\Vert  \sum_n   d_n e^{ 2\pi i x_n \cdot (.-k)}  \hat\psi(.-k) \right\Vert_{L^2((0,1]^d)}^2\leq B_1 C.
$$
In view of $\sum_{k\in\Z^d}\left\Vert  g(.-k) \right\Vert_{L^2((0,1]^d)}= ||g||_2$, we obtain
$$A_1 c \leq  \left\Vert \sum_n d_n e^{2\pi i x_n \cdot y}  \hat\psi \right\Vert_2^2\leq B_1 C.
$$
By Plancherel's theorem,   the latter is equivalent to $A_1 c \leq  \left\Vert \sum_n d_n  \tau_{x_n}\psi \right\Vert_2\leq B_1 C
$
and   so $\{\tau_{x_k}\psi \}_{k\in\Z^d}$  is a Riesz basis of $V^2(\psi, X) $, as required.    \end{proof}

\begin{proof}[Proof of Theorem \ref{T-non-bandlimited}] Let  $\B =\{e^{2\pi i x\cdot x_n}\}_{n\in\Z^d}$ be a
  Riesz  basis in $L^2([0,1)^d)$;  it is proved in   \cite{PW} (see also Section 2.3 in \cite{KN})     that  we can  find $\delta>0$ such that, if    $Y=\{y_j\}_{j\in\Z^d}\subset\R^d$ satisfies   $\sup_j|x_j-y_j|_2<\delta $,   then also $\{e^{2\pi i y_j\cdot x}\}_{j\in\Z^d} $ is a   Riesz basis in $L^2([0,1)^d)$.
  By Lemma \ref{L-bases-amal},  $\{\tau_{y_n}\psi \}$  is a Riesz basis of $V^2(\psi, Y) $.
  \end{proof}

  \medskip
  \noindent
  {\it Remark.}
When $Y=\{y_k\}_{k\in\Z^d}$ is such that  $\sup_{k\in\Z^d}|k-y_k|_\infty<\frac 14$,  by the multi-dimensional generalization of Kadec's theorem proved in \cite{SZ} we have that $ \{e^{2\pi i y_j\cdot x}\}_{j\in\Z^d} $ is a   Riesz basis in $L^2([0,1)^d)$ and
  by Lemma \ref{L-bases-amal},  $\{\tau_{y_n}\psi \}_{n\in\Z^d}$  is a Riesz basis of $V^2(\psi, Y)\}$.

 \subsection{Proof of Theorem \ref{T-stab-Vpsi-bis}}
In order to prove  Theorem \ref{T-stab-Vpsi-bis} we need some preliminary result: first, we prove the following
 \begin{lemma}\label{L-const-Vp} Let $(a,b)\subset\R$,  with $a<b<\infty$, and let $1\leq p < \infty$. Let
  $\psi\in L^p_0(a,b)$.
 For every finite set of coefficients   $\{\alpha_j\}\subset\C$,   we have that
 $$ \left\Vert\sum_k \alpha_k \tau_k\psi \right\Vert_p^p \leq  \|\psi\|_p^p([b-a]+1)^{p-1} \sum_k  |\alpha_k|^p
 $$
where $[\ ]$ denotes the  integer part.
 \end{lemma}
 \begin{proof}
 For simplicity we  let $a=0$. When $b\leq 1$ the supports of the $\tau_k\psi  $'s are disjoint and
 so $\|f\|_p^p=\left\Vert\sum_k \alpha_k \tau_k\psi \right\Vert_p^p =\|\psi\|_p^p\sum_k  |\alpha_k|^p
 $. When $b>1$ the supports of the $\tau_k\psi $ overlap, and there are at most $ [b]+1$  of such supports  that intersect at each point. By the elementary inequality
$ \left(x_1+\dots+x_m\right)^p\leq m^{p-1} \left(x_1^p+\dots+x_m^p\right)$  (which is valid when  the  $x_j $  are non-negative) we have that
 $$
  |f(t)|^p  =  |\sum_{k}   a_k \tau_k\psi(t) |^p   \leq ([b]+1)^{p-1} \sum_{k}  |a_k|^p|\tau_k\psi(t)|^p
 $$
 and so  $\|f\|_p^p\leq ([b]+1)^{p-1} \|\psi\|_p^p\sum_k  |\alpha_k|^p$
 as required.
 \end{proof}

 \medskip

Let   $Y= \{y_k\}_{k\in\Z^d}$   be a  discrete subset of $\R^d$.
%y {\it discrete} we mean that   there exists $\delta>0$ such that $ B(y_k ,\delta) \cap B(y_j,\delta)=\emptyset$ whenever $k\ne j$.
Assume that   $L=\sup_{k\in\Z^d}|{ y_{k}}-k|_2<\infty$. We prove the following

 \begin{lemma}\label{L-stab-Vpsi}
Let $D=\prod_{j=1}^d [a_j,b_j] $  and let $\psi \in W^{1,p}_0(D)$. Then,
 for every finite set of coefficients   $\{\alpha_j\} \subset\C$  such that $\sum_k |\alpha_k|^p=1$,  we have that
  \begin{equation}\label{e11}\left\Vert\sum_k \alpha_k(\tau_k \psi -\tau_{y_k}\psi) \right\Vert_p^p      \leq L\sum_{j=1}^d(1+[b_j-a_j +L])^{p-1}   \|\partial_j\psi \|_p^p.
\end{equation}
\end{lemma}

\begin{proof} When  $d=1$ and $D=(a,b)$, we prove that
\begin{equation}\label{e1}\left\Vert\sum_k \alpha_k(\tau_k \psi -\tau_{y_k}\psi )\right\Vert_p^p      \leq L(1+[b-a +L])^{p-1}   \|\psi'\|_p^p
\end{equation}
 where $\psi'(t) $ denotes the distributional derivative of $\psi$.
 Assume first  that $y_k>k$.
Observing  that
$
 \psi (t+y_k)-  \psi(t+k)   =\int_k^{y_k}\psi'( t+x )dx
$ and that $ |k-y_k|\leq L$,
we have that
$$\left\Vert\sum_k \alpha_k(\tau_k \psi -\tau_{y_k}\psi)  \right\Vert_p^p =\left\Vert\sum_k \alpha_k\int_{t+k}^{t+y_k}\psi'(x )dx \right\Vert_p^p  $$
$$
\leq \left\Vert\sum_k|\alpha_k|\int_{t+ k}^{ t+k+L}|\psi'(x )|dx \right\Vert_p^p  =   \left\Vert\sum_k|\alpha_k|\tau_k g\right\Vert_p^p
$$
where we have let $g(t)=\int_{t }^{ t +L }|\psi'(x )|dx $.
It is easy to verify that  $g(t)$ is supported in the interval $[a-L, b]$. Indeed,
 $\psi'  $ is supported in    $[a,b]$ and so   $g(t)  \equiv 0$   whenever
$ t +[0,L]\cap[a,b]=\emptyset$. Thus,  $g(t)\equiv 0$ when  $t+L<a$ or $t >b$, or: $g(t)\equiv 0$ when    $t\in \R-[a-L, b  ]$, as required.

By Lemma \ref{L-const-Vp}
\begin{equation}\label{2}
\left\Vert\sum_k \alpha_k(\tau_k\psi -\tau_{y_k}\psi )\right\Vert_p^p\leq
\left\Vert\sum_k |\alpha_k| \tau_k g\right\Vert_p^p \leq (1+[b-a +L])^{p-1} \|g\|_p^p\, .
\end{equation}
We apply a change of variables  and Minkowsky's  integral inequality; we gather
$$
\|g\|_p  = \left\Vert \int_{t }^{ t +L}|\psi'(x )|dx\right\Vert_p =\left\Vert \int_{0 }^{  L}|\psi'(x+t )|dx\right\Vert_p
$$$$
\leq L \|\psi'\|_p
$$
which together with the inequality \eqref{2} concludes the proof of \eqref{e1}. When $y_k<k$  the proof if similar, but    the function $g(t)  $ defined above should be replaced by  $g(t)=\int_{t }^{ t -L }|\psi'(x )|dx $, a  function  supported   in the interval $[a, b+L]$.

When $d=2$ we can let $y_k=(y_{k,1}, y_{k,2})$  and $k=(k_1,\,   k_2)$  and write
$$
\left\Vert\sum_k \alpha_k (\tau_k \psi- \tau_{y_k}\psi)  \right\Vert_p  $$$$\leq \left\Vert\sum_k \alpha_k(\tau_{(k_1, k_2)} \psi -\tau_{(y_{k, 1}, k_2) }\psi) \right\Vert_p + \left\Vert\sum_k \alpha_k(\tau_{(y_{k, 1}, k_2) }\psi -\tau_{(y_{k, 1}, y_{k,2}) } \psi)  \right\Vert_p
$$
$$
= \left\Vert\sum_k \alpha_k(\tau_{ k_1 }  \psi_1 -\tau_{ y_{k, 1}  }\psi_1 ) \right\Vert_p +
\left\Vert\sum_k \alpha_k(\tau_{ k_2  }\psi_2 -\tau_{  y_{k,2}  }\psi_2) \right\Vert_p
$$
where we have let $\psi_1= \tau_{(0,k_2)}\psi$ and $\psi_2= \tau_{(y_{k_1}, 0) }\psi$.  The inequality  \eqref{e1}, applied to $\psi_1$ and $\psi_2$, yields \eqref{e11}.
The case $d>2$ is similar.

\end{proof}

\medskip
\begin{proof}[Proof of Theorem \ref{T-stab-Vpsi-bis}]
 Follows  from Lemmas \ref{L-PW} and \ref{L-stab-Vpsi}. \end{proof}

\subsection{  $\rect$ and step functions}

Since    Sobolev spaces $W^{1,p}(\R)$ do  not contain discontinuous functions,  we cannot apply
 Theorem \ref{T-stab-Vpsi-bis}  when $\psi$ is a step function.
%In this sub-section we prove  Theorem \ref{T-step}  and some corollary.

\medskip
 Let $\psi=\rect$;
it is immediate to verify that, for every   $1\leq p<\infty$, the set  $\{\tau_{j}\rect \}_{j\in\Z }$ is a $p-$Riesz basis of $ V^p(\rect)$ with frame constants $A=B=1$.  Throughout this section we let $Y=\{y_k\}_{k\in\Z }\subset \R$, with $L=\sup_{k\in\Z^d} |y_k-k| $ and we assume $1\leq p<\infty$.

Lemma \ref{L-stab-Vo-Lp}  below is an easy generalization of Lemma 10 in \cite{DV1}.

   \begin{lemma}\label{L-stab-Vo-Lp}
 Assume $0\leq L<1$. For every finite set of coefficients   {$\{\alpha_k\}_{n\in\Z^d}\subset\C$}   we have that
   \begin{equation}\label{e-cond-PW-p}\left\Vert\sum_k \alpha_k(\rect(t- k)-\rect(t- y_k))\right\Vert_p^p  < 2^pL   \sum_k  |\alpha_k|^p.
\end{equation}
\end{lemma}

\begin{proof}
Assume $ \sum_k  |\alpha_k|^p=1$.
Let
 {\begin{equation}\label{e-sum-In}
f(t)=\sum_k \alpha_k\left(\rect(t-  k)-\rect(t- y_k)\right)= \sum_k \alpha_k \chi_{I_k},
\end{equation}}
where $I_j $ denotes the support  of $  \rect (t- j)-\rect (t- y_j)$.
  When $ y_j\ne  j$, $I_j$ is union of two intervals that we  denote with $I_j^+$ and $I_j^-$. When  $y_j>j$,   we let
 $$ I_j^-=( j-\frac 12, \  y_j-\frac 12),\quad  I_j^+=( j+\frac 12, \  y_j+\frac 12).
  $$
We   use (improperly)  the same notation to denote $I_{j}^{+}$  and $I_j^-$ also when  $y_j<j$.

Since we have assumed $|y_h-h|\leq L<1$, for every given interval {$J=I_h^{\pm}$} there is at most another interval $I_k^{\pm}$ that overlap with  $J$; thus, for every $t\in\R$, the sum in \eqref{e-sum-In} has at most $2$ terms.
By the elementary inequality
$ \left(x_1+\dots+x_m\right)^p\leq m^{p-1} \left(x_1^p+\dots+x_m^p\right)$    we have that $|f(t)|^p \leq 2^{p-1} \sum_k |\alpha_k|^p \chi_{I_k}(t)$, and
$ \|f\|_p^p\leq 2^{p-1}\sup_k |I_k|= 2^{p-1}(2L) =2^p\,L$
and
the proof of   the Lemma   is concluded.
 \end{proof}

Lemma  \ref{L-stab-Vo-Lp} and   Lemma   \ref{L-PW} yield the following

 \begin{theorem}\label{T-rect-X}
With the notation of Lemma \ref{L-stab-Vo-Lp}, the set $\{\tau_{y_k}\rect \}_{k\in\Z }$ is a $p-$Riesz basis in $V^p(\rect;\, Y)$ if $2^p L <1$.
 \end{theorem}

 \medskip

\begin{corollary}\label{C-f*rect} Let $\psi_0\in L^1(\R)$ and let $\psi=\rect*\psi_0$.  Suppose that    $ \{\tau_k\psi \}_{k\in\Z }$ is a $p-$Riesz basis for $V^p(\psi)$.
 For every   finite set of coefficients $ \{ \alpha_{k} \}_{n\in\Z }\subset \C$  with $\sum_k|\alpha_k|^p=1$, we have that
 %such that ${\sum_{k} |\alpha_{k}|^2}=1$, we have that
$$
\left\Vert\sum_k \alpha_k(\psi(t- k)-\psi(t- y_k))\right\Vert_p^p  < 2^p L  \|\psi_0\|_1^p
$$
and the set $\{\psi(t-y_k)\}_{k\in\Z }$  is a $p-$Riesz basis for   for $V^p ( \psi; Y )$ whenever
\begin{equation}\label{eq:pv-p2}
2^pL\|\psi_0\|_1^p   < A.\,
\end{equation}
 \end{corollary}

 \medskip
\noindent
{\it Remark.} If  $\hat\psi_0(x)\ne 0 $  for every $x\in\R$, then   the set $ \{\tau_k\psi \}_{k\in\Z }$ is a $p-$Riesz basis for $V^p(\psi)$. Indeed,   $ \sum_{m \in\Z  } |\widehat \rect(y+m)|^2 =\sum_{m \in\Z  }  |\sinc(x+k)|^2\ne 0$ whenever $x\in [-\frac 12, \frac 12)$ and so also  $\sum_{m \in\Z  } |\hat\psi(x+k)|^2 = $ $\sum_{m \in\Z  }|\hat\psi_0(x+k)\widehat{\rect}(x+k)|^2\ne 0$; by Lemma   \ref{L-eq-cond-V2}   the set $\{\tau_k\psi \}_{k\in\Z }$ is a $p-$Riesz basis for $V^p(\psi)$.

\begin{proof}[proof of Corollary \ref{C-f*rect}]
Let
 $$ F (t) = \sum_ {k}  \alpha_{ k}\left(\psi(t- k)-\psi(t- y_k)\right), \quad
     f (y)=\sum_{k\in\Z  } \, \alpha_{ k}   \left(\rect( y- k )\!-\! \rect ( y- {x_{k}}  \right)$$
and we show that
 $
 \|F\|_p^p\leq 2^p L\|\psi_0\|_1^p .
$
We gather
 \begin{align*}
F (t) &= \int_{-\infty}^\infty\!\!\!  \psi_0(t-y)  \! {{\sum_{k\in\Z }}\, \alpha_{ k}   \left(\rect( y- k )\!-\! \rect ( y- {x_{k}}  )\right)dy}
\\ &
=  \psi_0 * f (t).
\end{align*}
Thus, by  Young's inequality and Lemma \ref{L-stab-Vo-Lp},
$$
\|F\|_p^p\leq \|\psi_0\|_1 ^p\|f \|_p^p \leq  2^pL \|\psi_0\|_1 ^p
$$
and the proof of the corollary is concluded.
\end{proof}

\medskip
Let $\beta_m(x)=\rect^{(m+1)}$ be the B-spline of order $m>1$.  We recall that   $\beta_m$ is   supported in the interval $[-\frac {m+1}{2}, \frac {m+1}{2}]$ and $\beta_m(x) \in W^{1,p}(\R)$ whenever $m\ge 1$. It is easy to verify  by induction on $m$  that
$\|\beta^m\|_p\leq 1$ and  $\|\beta_m'\|_p\leq 2$.
It is  known that   $\{\tau_k\beta_m\} _{k\in\Z }$ is a  Riesz basis of $V^2(\beta_m)$ whose  Riesz constants  $A(m)$ and $B(m) $  are explicitly evaluated in \cite{M}. See also  \cite{SelRad16}.
By the observations after Lemma \ref{L-eq-cond-V2}, $\{\tau_k\beta_m\}_{k\in\Z }$ is a $p-$Riesz basis of $V^p(\beta_m)$ with   constants  $A_p(m)>0$ and $B_p(m)<\infty $.

We prove the following

\begin{corollary}\label{C-splines}
Assume that   $L <2^{-p}A_p(m)$. Then, the set $\{ \tau_{y_k}\beta_m \}_{k\in\Z }$  is a $p-$Riesz basis  of $V^p ( \psi, Y )$.
\end{corollary}
\begin{proof} We apply Corollary \ref{C-f*rect} with $\psi_0=\beta^{m-1}$.
\end{proof}

\medskip\noindent
{\it Remark.}
 We could have used  Theorem \ref{T-stab-Vpsi-bis} to prove Corollary \ref{C-splines}, but we would have obtained a lower upper bound for $L$ (namely,
 $   L <\frac{A_p(m)}{2(2+m)^{p-1}}$).

\medskip

\begin{proof}[proof of Theorem \ref{T-step}]

Let  $g(t)=\sum_{|j|\leq J} s_j\rect(   t-j)$.
Let $\{\alpha_{k}\}_{n\in\Z^d} \subset\C$ be a finite set of coefficients such that $\sum_{k} |\alpha_{k}|^p=1$. Let
$$f(t) =\sum_{k} \alpha_{k}\, \left(g\left(  t- k\right)-g\left( t-  x_{k}\right)  \right)\, .$$

As in previous theorems, we   find conditions on $L$ for which $\|f\|_p^p<A$. We have
\begin{align*}
f(t) &=
\sum_{|j|\leq J} s_j \sum_{k} \alpha_{k} \,\left( \rect(   t-j-  k )-\rect(   t-j-   x_{k} )  \right)
\\  &= \sum_{|j|\leq J} s_j f_j(t). \end{align*}
By Minkowski and H\"{o}lder inequalities, and noting that $\sum_{|j|\leq J} |s_j |^q=  \|g\|_q^q$, it follows that
\begin{align}\label{e-as2bis}
\|f\|_p \leq \sum_{|j|\leq J}|s_j | \|f_j\|_p&\leq \left(  \sum_{|j|\leq J} |s_j |^{p'}  \right)^{\frac 1{p'}}\left(\sum_{|j|\leq J}\|f_j\|_p^p\right)^{\frac 1p} \notag \\
                                             &=  \|g\|_{p'}\, \left(\sum_{|j|\leq J}\|f_j\|_p^p\right)^{\frac 1p}\, ,
\end{align}
With the change of variables   $ t-j=t'$ in the integral below, we obtain
$$
\int_\R |f_j(t)|^pdt= \int_\R\left|
\sum_{k} \alpha_{k} \left( \rect(   t-j- k )-\rect(   t-j- x_{k} )  \right)\right|^pdt
$$
$$
=
 \int_\R\left|
\sum_{k} \alpha_{k} \left( \rect(  t'- k )-\rect(  t'- x_{k} )  \right)\right|^pdt'
$$
$$
=
    \left\Vert\sum_k \alpha_k(\rect(t- k)-\rect(t- y_k)\right\Vert_p^p        \, .
$$
From Lemma \ref{L-stab-Vo-Lp}, follows that  the integral  above is
$ \leq  2^p  L      $.
 We gather:  $\|f\|_p^p \leq   2^p L J \, \|g\|_{p'}^{p } $.
By assumption $2^p L J \, \|g\|_{p'}^{p }<A$, and by Lemma \ref{L-PW} Theorem \ref{T-step} follows.
\end{proof}

\section{Remarks and open problems}

We have discussed  Problem 1  when $p=2$ and  the Fourier transform of the window function $\psi$ has compact support. When $\psi$ is not band-limited, Plancherel's theorem implies that the set $\{\tau_{x_j}\psi\}_{j\in\Z^d}$ is a Riesz basis in $V^2(\psi, X)$ if and only if the set ${\mathcal V}= \{e^{2\pi i x_j\cdot x}|\hat\psi|\}_{j\in\Z^d}    $   is a   Riesz sequence  in  $L^2(\R^d)$, and hence a Riesz basis in $V=\overline{\mbox{Span}({\mathcal V}))}$.   By a  theorem of Krein-Milman-Rutman  (see e.g. \cite[Theorem 11]{Y}) for every $j\in\Z^d$ there exists $\epsilon_j>0$ such that  every set of functions $\{g_j\}_{j\in\Z^d}\subset V$ is a  Riesz  basis of $V$ whenever $||g_j-e^{\pi i x_j\cdot x}|\hat\psi||_2<\epsilon_j$.    We can find $\delta_j>0$ such that
$||(e^{\pi i x_j\cdot x} -e^{\pi i y_j\cdot x})\hat\psi||_2<\epsilon_j$ whenever $|x_j-y_j|_2<\delta_j$, but we do not know whether  the $\delta_j$'s have a lower bound or not.

\medskip

For functions $\psi$ in $L^p(\R^d) $ for every $p\in[1,\infty)$
it would be interesting to prove    conditions that would ensure that a $q$-Riesz basis in $V^q(\psi, X)$ for some $q$ is automatically a $p$-Riesz basis  in $V^p(\psi, X)$ for all $p$.    Lemma \ref{L-eq-cond-V2} (b)  shows that, for certain $\psi$,   if the set $\{\tau_k\psi\}_{k\in\Z^d}$ is a  2-Riesz  basis of   $V^2(\psi)$, it is also  a  p-Riesz   in $V^p(\psi)$ but  the method of proof of this result does not generalize well to other sets of translations.  Results in \cite{AldBaska} and \cite{ShinSun} may help    generalize Lemma \ref{L-eq-cond-V2}.

\medskip
It would also be interesting to define and  investigate p-Riesz bases in quasi-shift invariant spaces $V^p(\psi, X)$ when  $0<p<1$.  Wavelet in $L^p$ with  $0<p<1$ have been considered in \cite {GHT}. We feel that the results contained in Section 3 of \cite {GHT} may  help  the  understanding  of  $V^p(\psi, X)$ when  $0<p<1$.


\begin{thebibliography}{9}

 %\bibitem{AAK} E. Acosta-Reyes, A. Aldroubi, and I. Krishtal. {\it On stability of sampling-reconstruction models}, Adv. Comp. Math.   31 (2009), pp. 5--34.

%\bibitem {At} N. D. Atreas, {\it On a class of non-uniform average sampling expansions and partial reconstruction in subspaces of $L^2(\R)$}, Adv. Comput. Math. 36(1) (2012), 21--38.
%
\bibitem{AS} A.  Aldroubi, Q. Sun, {\it Connection between $p-$frames and $p-$Riesz bases
in locally finite SIS of $L^p(\R)$ }
  Proceedings of SPIE - The International Society for Optical Engineering,  February 1970
  %
\bibitem{AST} A. Aldroubi, Q. Sun, W. Tang, {\it p-Frames and shift invariant subspaces of$L^p$} , J. Fourier Anal. Appl. 7 (2001) 1–-21.

\bibitem{AldBaska}
A. Aldroubi, A. Baskakov, I. Krishtal, {\it Slanted matrices, Banach frames, and sampling}, J. Funct. Anal. 255 (2008) 1667--1691.

\bibitem{B}   H. Brezis,
{\it Functional Analysis,
Sobolev Spaces and Partial
Differential Equations}, Springer Verlag 2011.
%
\bibitem{BR}
A. Ben-Artzi, A. Ron, {\it On the Integer Translates of a Compactly Supported Function: Dual Bases and Linear Projectors}, SIAM J. Math. Anal., 21(6), 1550–1562.
%
\bibitem{CCS} P. Casazza, O. Christensen, D.T. Stoeva, {\it Frame expansions in separable Banach spaces}, J. Math. Anal. Appl. 307 (2005) 710--723.

\bibitem{CS} O. Christensen, D.T. Stoeva, {\it p-Frames in separable Banach spaces}, Adv. Comput. Math. 18 (2003) 117–-126.

\bibitem{DV1}
L. De Carli, P. Vellucci, {\it Stability theorems for the n-order hold model}, 	arXiv:1605.01706  (2016) (submitted)

\bibitem{F} H. Feichtinger,   {\it Spline-type spaces in Gabor analysis}. Wavelet analysis (Hong Kong, 2001), 100–122, Ser. Anal., 1, World Sci. Publ., River Edge, NJ, 2002.

\bibitem{FMR} H. Feichtinger,  U. Molter, J.L.Romero, {\it Perturbation techniques in irregular spline-type spaces}, World Scientific Publishing Co. Inc., Int. J. Wavelets Multiresolut. Inf. Process, 6 (2) (2008)  249–277.

\bibitem{FO} H. Feichtinger, H.; D.M. Onchiş, {\it Constructive reconstruction from irregular sampling in multi-window spline-type spaces}. Progress in analysis and its applications, 257–265, World Sci. Publ., Hackensack, NJ, 2010.

\bibitem{GHT} G. Garrigos;  R. Hochmuth; A. Tabacco, {\it Wavelet characterizations for anisotropic Besov
spaces with $0 < p < 1$}
Proc.   Edinburgh Math. Soc. 47 (2004) 573–-59.

\bibitem{GS} K. Gr\"ochenig, J. St\"ockler, {\it Gabor Frames and Totally Positive Functions}, Duke Math. J.
162(6) (2013), 1003--1031.

\bibitem{HL} K. Hamm, J. Ledford, {\it On the structure and interpolation properties  of quasi-invariant shift spaces},   arxiv:1703.01533   (2017)

 \bibitem{He10} C. Heil, {\it A basis theory primer},  Appl. Num. Harm. Analysis, Birkh\"auser  2011.
 %
\bibitem{JM}   R.Q. Jia,  C. A. Micchelli, {\it Using the refinement equation for the construction of pre-wavelets II: power
of two}, In  "Curve and Surface" (P. J. Laurent, A. Le Mehaute and L. L. Schumaker eds.), Academic Press,
New York 1991, pp. 209--246.


\bibitem{K} M.I. Kadec, {\it The exact value of the Paley-Wiener constant}, Soviet Math. Dokl., 5
(1964), 559--561.
%
\bibitem{KN} G. Kozma, S. Nitzan, {\it  Combining Riesz bases}, Inv. Math.
  199 (1) (2015), , pp 267–-285
%
\bibitem{M}  E.V.  Mischenko, {\it Determination of Riesz bounds for the spline basis with the help of trigonometric
polynomials}. Sib. Math. J. 51(4), 660–-666 (2010)
%
\bibitem{Pavlov}  B. Pavlov,  {\it Basicity of an exponential system and Muckenhoupt's condition}, Soviet Math.
Dokl. 20 (1979)  655--659.
 \bibitem{PBP} Prautzsch, H., Boehm, W., Paluszny, M., {\it Bezier and B-Spline Techniques}, Springer Science and Business Media (2002).
%
\bibitem{PW}  R. Paley and N. Wiener, {\it Fourier transforms in the complex domain}.
Amer. Math. Soc. Colloquium Publications, vol. 19; Amer. Math. Soc., New
York, 1934.
%
\bibitem{Ro} J.L. Romero, {\it Explicit localization estimates for spline-type spaces}. Sampl. Theory Signal Image Process. 8 (2009), no. 3, 249–-259.
\bibitem{R}
A. Ron, {\it A necessary and sufficient condition for the linear independence of the integer translates of a compactly supported distribution}, Constructive Approximation
  5 (1),   297–-308.
  %
\bibitem{Russo}
A. M. Sedletskii, Izv. Vyssh. Uchebn. Zaved., Mat., No. 7, 85–-91 (1973).
%
\bibitem{Sed16}
A. M. Sedletskii, Equivalence of the trigonometric system and its perturbations in $L^p (− \pi, \pi)$, Doklady Mathematics. Vol. 94. No. 1. Pleiades Publishing, 2016.
%
\bibitem{Sch2} I. J. Sch\"onberg,  {\it  Cardinal interpolation and spline functions.} J.   Approx. teory  2 (2) (1969), pp. 167--206.

%\bibitem{SelRad15}A. Antony Selvan, R. Radha, {\it Sampling and reconstruction in shift invariant spaces on $\mathbb {R}^{d}$.}  Ann. Mat. Pura Appl.(4) 194.6 (2015)  1683--1706.
 %
 \bibitem{SelRad16} A. Antony Selvan, R. Radha, {\it Sampling and Reconstruction in Shift Invariant Spaces
of B-Spline Functions}  Acta Appl. Math.
DOI 10.1007/s10440-016-0053-6 (2016)

\bibitem{ShinSun}
C. E. Shin, Q. Sun, {\it Stability of localized operators}, J. Funct. Anal. 256 (2009) 2417--2439.


 \bibitem{S} Q. Sun, {\it Stability of the Shifts of Global Supported Distributions}
 J.Math. Analysis and Appl.
  261 (1)(2001)   113--125.
\bibitem{SZ} W. Sun,  X. Zhou, {\it On the stability of multivariate trigonometric systems}.
J. Math Anal. Appl. 235 (1999), 159--167.
\bibitem{UAE} M. Unser, A. Aldroubi, M. Eden, {\it B-spline signal processing. I. Theory}, Signal Processing, IEEE Transactions  on 41 (2), 821--833.

\bibitem{Y} R. M. Young, {\it An introduction to nonharmonic Fourier series}, Academic Press, 2001



\end{thebibliography}
 \end{document}